\documentclass{article}

\usepackage[numbers]{natbib}
\usepackage{blindtext}
\usepackage{lineno,hyperref}
\modulolinenumbers[5]
\usepackage{graphicx} 
\usepackage{subfigure}
\usepackage{tikz}
\usepackage{multirow}
\usetikzlibrary{positioning}
\usepackage{colortbl}
\usepackage{amsmath,amssymb}  
\usepackage{chngcntr}
\usepackage{epstopdf}
\usepackage{lipsum}   
\usepackage{setspace}
\usepackage{nameref}
\usepackage{float}
\usepackage{stackengine}
\usepackage{makecell}
 \usepackage{eurosym}
\usetikzlibrary{bayesnet}
\usepackage{natbib}
\usepackage{xurl}
\usepackage{adjustbox} 
\usepackage{tikz}
\usepackage{pgfplotstable,colortbl}
\usepackage{xcolor}
\pgfplotsset{compat=1.12}
\usepgfplotslibrary[colormaps] 
\usetikzlibrary{pgfplots.colormaps} 
\usepackage{pgfplotstable}
\usepackage{pgfplots}
\usepackage{authblk}
\usepackage[T1]{fontenc}
\usepackage[utf8]{inputenc}
\usepackage{lmodern} 
\usepackage{lineno}
\pgfplotsset{compat=1.17}

\newtheorem{theorem}{Theorem}

\newtheorem{corollary}[theorem]{Corollary}

\newtheorem{definition}[theorem]{Definition}

\newenvironment{proof}[1][Proof]{\noindent\textbf{#1.} }{\ \rule{0.5em}{0.5em}}

\pgfplotstableset{
  /color cells/min/.initial=0,
  /color cells/max/.initial=1000,
  /color cells/textcolor/.initial=,
  color cells/.style={
    postproc cell content/.append code={%
          \pgfkeys{/color cells/.cd,#1}%
      \pgfkeysgetvalue{/pgfplots/table/@preprocessed cell content}\value%
      \ifx\value\empty%
      \else%
     
     \pgfmathparse{int(less(##1,-2))} 
       \ifnum\pgfmathresult=1
           \edef\temp{\noexpand\pgfkeyssetvalue{/pgfplots/table/@cell content}{%
            \noexpand\cellcolor{gray}{n/a}%
          }%
        }%
        \temp%
      \else%
      \pgfmathfloatparsenumber{\value}\pgfmathfloattofixed{\pgfmathresult}%
        \let\value=\pgfmathresult%
        \pgfplotscolormapaccess%
            [\pgfkeysvalueof{/color cells/min}:\pgfkeysvalueof{/color cells/max}]%
            {\value}{\pgfkeysvalueof{/pgfplots/colormap name}}%
        \pgfkeysgetvalue{/pgfplots/table/@cell content}\typesetvalue%
        \pgfkeysgetvalue{/color cells/textcolor}\textcolorvalue%
        \toks0=\expandafter{\typesetvalue}%
        \edef\temp{\noexpand\pgfkeyssetvalue{/pgfplots/table/@cell content}{%
            \noexpand\cellcolor[rgb]{\pgfmathresult}%
            \noexpand\definecolor{mapped color}{rgb}{\pgfmathresult}%
            \ifx\textcolorvalue\empty\else\noexpand\color{\textcolorvalue}\fi%
            \the\toks0%
          }%
        }%
        \temp%
      \fi%
      \fi%
      }%
  }%
}%

\begin{document}

\title{Analyzing Smoothness and Dynamics in an SEIR$^{\text{T}}$R$^{\text{P}}$D
Endemic Model with Distributed Delays}
\author[1]{Tin Nwe Aye}
\author[2]{Linus Carlsson}

\affil[1]{Department of Mathematics, Linn\'{e}universitetet, Växjö, Sweden}
\affil[2]{Division of Applied Mathematics, Mälardalen University, Box 883, 721 23, Västerås, Sweden}
\maketitle

\begin{abstract}
This article explores the properties of an SEIR$^{\text{T}}$R$^{\text{P}}$D
endemic model expressed through delay-differential equations with distributed
delays for latency and temporary immunity. Our research delves into the
variability of latent periods and immunity durations across diseases, in
particular, we introduce a class of delays defined by continuous integral
kernels with compact support. The main result of the paper is a kind of
smoothening property which the solution function posesses under mild
conditions of the system parameter functions. Also, boundedness and
non-negativity is proved. Numerical simulations indicates that the continuous
model can be approximated with a discrete lag endemic models. The study
contributes to understanding infectious disease dynamics and provides insights
into the numerical approximation of exact solution for different delay scenarios.

\end{abstract}

\section{Introduction}

The SEIR$^{\text{T}}$R$^{\text{P}}$D endemic model was initially proposed by
\cite{disease2024} to address epidemics spread through contact with infectious
individuals, incorporating single time delays for the latency period and
temporary immunity. In this paper, our focus is on analyzing the properties of
an SEIR$^{\text{T}}$R$^{\text{P}}$D endemic model expressed through
delay-differential equations, incorporating distributed delays for both
latency and temporary immunity, as opposed to the original single time delays.
Several research articles have explored similar realistic phenomenas,
incorporating discrete or distributed time delays for both discrete and
continuous models \cite{guglielmi2022delay, zhu2020dynamics,
shayak2020temporary, wei2008delayed, zhang2008analysis, beretta1995global}.

In this study, we develop and modify the SEIR$^{\text{T}}$R$^{\text{P}}$D
endemic model proposed by \cite{disease2024} to analyze the properties of the
solution within a space of continuous functions. Emphasizing the variability
in the duration of latent periods and temporary immunity among various
diseases, which has been confirmed emperically, see \cite{xin2022estimating,
behr2018revisiting, duintjer2005dynamic}, and even within individuals with the
same disease, see \cite{abufares2022covid, moosazadeh2015incidence,
hall1991immunity}. Our research delves into the analysis of disease
transmission by examining variations in time delays, cf. \cite{arino2006time,
al2016time, misra2013mathematical}.

The study employs distributed delay kernel functions to model the force of
infection from the exposed stage to the infectious stage, utilizing the
Lebesgue integrable probability density functions, $\Phi$, with compact
support on the positive real axis, that is%
\[
\int_{\mathbb{R_{+}}}\Phi(\rho)d\rho=1.
\]
We also consider distributed time delays for temporary recovery from the
disease with the same mathematical properties as the immunity function.

To ensure that the model accurately reflects the underlying system's behavior,
we analyze the properties of the solution of the endemic model. Our primary
goal is to enhance crucial aspects such as boundedness and non-negativity, as
solution that become unbounded or negative is not meaningful in the context of
the system being modeled. Additionally, we ensure smoothness of solution by
demonstrating increased smoothness as time progresses.

In the numerical simulation section, the solution to the SEIR$^{\text{T}}%
$R$^{\text{P}}$D endemic model with continuous kernel functions are compared
with analougous discrete lag endemic models. The simulation results indicates
that the approximated solution with discrete lags converges to the exact
solution as the number of lag points are increased.

The structure of this paper is as follows: In Section \ref{sec:endemicModel},
we present the system of an endemic model with the essential underlying
assumptions. In Section \ref{sec:mathematicalAnalysis}, we examine the
solution of the endemic model to ensure non-negativity, boundedness, and
increasing smoothness. To validate the analytical findings and enhance the
model complexity, Section \ref{sec:numericalSimulation} presents the results
of numerical simulations of the continuous time delay and discrete lag endemic
models. The paper ends with a discussion of the findings in the paper.

\section{An endemic dynamics model with distributed
delays\label{sec:endemicModel}}

The SEIR$^{\text{T}}$R$^{\text{P}}$D endemic model serves as a compartmental
model for comprehending the propagation of a disease within an unstructured
population. The population is stratified into six compartments: susceptible
($S$), exposed ($E$), infected ($I$), temporary recovered ($R^{\text{T}}$),
permanently recovered ($R^{\text{P}}$), and disease death ($D$). The dynamical
behavior of the SEIR$^{\text{T}}$R$^{\text{P}}$D model is encapsulated by a
system of delayed differential equations, as outlined in System
\ref{model:mainGenNoBirthOrNaturalDeath}.
\begin{align}
\frac{dS}{dt}  &  =-\beta(t)I(t)S(t)+p\gamma\int_{\mathbb{R_{+}}}I(t-\rho
)\Phi(\rho)d\rho,\nonumber\\
\frac{dE}{dt}  &  =\beta(t)I(t)S(t)-\int_{\mathbb{R_{+}}}\beta(t-\tau
)I(t-\tau)S(t-\tau)\Psi(\tau)d\tau,\nonumber\\
\frac{dI}{dt}  &  =\int_{\mathbb{R_{+}}}\beta(t-\tau)I(t-\tau)S(t-\tau
)\Psi(\tau)d\tau-\gamma I(t)-\mu
I(t),\label{model:mainGenNoBirthOrNaturalDeath}\\
\frac{dR^{\text{T}}}{dt}  &  =p\gamma I(t)-p\gamma\int_{\mathbb{R_{+}}%
}I(t-\rho)\Phi(\rho)d\rho,\nonumber\\
\frac{dR^{\text{P}}}{dt}  &  =(1-p)\gamma I(t),\nonumber\\
\frac{dD}{dt}  &  =\mu I(t).\nonumber
\end{align}

Within this framework, the transition from susceptible ($S$) to exposed ($E$)
status occurs through contact with infectious individuals ($I$). The force of
infection is defined by $\beta(t)I(t)S(t)$, where $\beta(t)S(t)$ represents
the relative rate of infection transmission upon interaction with infected
individuals. The contact rate ($\beta$) is a metric encompassing both the
frequency of interactions and the infectiousness of individuals. The latency
time for an exposed individual to become infectious is probabilistic, as
expressed by:%
\begin{equation}
\int_{\mathbb{R_{+}}}\beta(t-\tau)I(t-\tau)S(t-\tau)\Psi(\tau)d\tau,
\label{eq:kernelTau}%
\end{equation}
where $\Psi(\tau)$ denotes the probability density function for the transition
of an exposed individual to the infectious state \footnote{Notably, the time
delays in disease systems are often discrete. For instance, in
\cite{disease2024}, the transition from exposed to infectious is characterized
by a factor $\beta(t-\tau)I(t-\tau)S(t-\tau)$ for some constant $\tau_{0}>0$.
In contrast, we assume the duration time of exposure to be random.}. Moreover,
for some $0<\theta<L<\infty$, $\Psi(\tau)=0$ if $\tau\notin\lbrack\theta,L]$,
i.e.
\begin{equation}
\operatorname*{supp}(\Psi)\subset\lbrack\theta,L]. \label{Constant:ThetaAndL}%
\end{equation}
Surviving infected individuals acquire either temporary immunity
($R^{\text{T}}$) or permanent immunity ($R^{\text{P}}$). The duration of
temporary immunity is assumed to be probabilistic, governed by a probability
density function $\Phi$. After this duration, individuals revert to the
susceptible state, and the transition rate is given by:
\begin{equation}
p\gamma\int_{\mathbb{R_{+}}}I(t-\rho)\Phi(\rho)d\rho, \label{eq:kernelRho}%
\end{equation}
where the constant $\gamma$ represents the recovery rate. The value of $p$
accounts for the proportion of recovering individuals transitioning from the
infectious stage to the temporary recovered stage. Furthermore, we have for
some $0<\epsilon<M<\infty$, $\Phi(\rho)=0$ if $\rho\notin\lbrack\epsilon,M]$,
i.e.%
\begin{equation}
\operatorname*{supp}(\Phi)\subset\lbrack\epsilon,M]. \label{Constant:Epsilon}%
\end{equation}
Infected individuals experience a constant mortality rate $\mu$, implying that
if not recovered, the individual is transferred to the disease death ($D$) compartment.

The additional assumptions on the model are:

\begin{itemize}
\item[(i)] The birth rate and death rate of the population are considered to
be zero in this endemic model.

\item[(ii)] The contact rate $\beta(t)$ is assumed to be a positive and smooth function.

\item[(iii)] The history data for each compartments are assumed to be positive constant.
\end{itemize}

\section{Mathematical analysis of the model\label{sec:mathematicalAnalysis}}

In this section, we examine the solution of System
\ref{model:mainGenNoBirthOrNaturalDeath} to assess their non-negativity,
boundedness, and smoothness. To determine the smoothness of solution, we
initially investigate the continuity and differentiability of the integral
terms in the System \ref{model:mainGenNoBirthOrNaturalDeath} in its genenral
setting. With the help of the properties these integrals, the increasing
smoothness of solution are proved. From now on, we assume that System
\ref{model:mainGenNoBirthOrNaturalDeath} satisfies that the contact rate
$\beta$ is a non-negative smooth function on $\mathbb{R}$, the history data is
constant, i.e.,%
\begin{align}
S(s)  &  =c_{S}>0,\text{ for all }s\leq0,\label{eq:HS}\\
I(s)  &  =c_{I}>0,\text{ for all }s\leq0, \label{eq:HI}%
\end{align}
and finally that the initial data is%
\begin{align}
E(0)  &  =\beta_{0}c_{I}c_{S}\int_{\theta}^{L}\Psi(\tau)\tau d\tau
\text{,}\label{eq:Es}\\
R^{\text{T}}(0)  &  =c_{I}p\gamma\int_{\epsilon}^{M}\Phi(\rho)\rho
d\rho\text{,}\label{eq:Rs}\\
R^{\text{P}}(0)  &  =(1-p)\gamma c_{I}\int_{\theta}^{L}\Psi(\tau)\tau
d\tau\text{.} \label{eq:RPs}%
\end{align}

\subsection{Positivity and boundedness of
solution\label{sec:positivityAndBoundedness}}

We first prove positivity of the solution in System
\ref{model:mainGenNoBirthOrNaturalDeath} by utilizing the positive constant
values for historical data of susceptible and infectious individuals. We end
this section by proving boundedness of the solution to System
\ref{model:mainGenNoBirthOrNaturalDeath}.

\begin{theorem}
[Non-negativity]\label{Thm:Non-negativity}Assume that the parameters $\mu$ and
$\gamma$ in the System~\ref{model:mainGenNoBirthOrNaturalDeath} are
non-negative constants. The contact rate, $\beta(t),$ is non-negative for all
$t$. Let the history data for be given by equations \eqref{eq:HS} and
\eqref{eq:HI} and the initial data be equations \eqref{eq:Es}--\eqref{eq:RPs}. Then,
$S(t)$, $E(t),I(t)$, $R^{\text{T}}(t),$ and $R^{\text{P}}(t)$ are non-negative
for all $t>0$.
\end{theorem}

\begin{proof}
The proof of the positivity of $I$ and $S$ is performed by showing positivity
with the help of mathematical induction over time windows $n\tau\leq
t\leq(n+1)\tau$ where $n=0,1,\dots$. \newline The base case, $n=0:$%
\newline\textbf{Positivity of }$I$:\newline By assumption, $I(t-\tau)>0$ and
$S(t-\tau)>0$ for all $0\leq t\leq\tau$. Consider the dynamics for $I$ on this
interval
\begin{equation}
\frac{dI}{dt}=\int_{\mathbb{R_{+}}}\beta(t-\tau)I(t-\tau)S(t-\tau)\Psi
(\tau)d\tau-\gamma I(t)-\mu I(t). \label{eq:Idynamic}%
\end{equation}
By the assumption, we get
\[
\frac{dI}{dt}\geq-(\gamma+\mu)I(t).
\]
By multiplying with the integrating factors for both sides, we get
\[
\frac{dI}{dt}e^{^{\int_{0}^{t}\left(  \gamma+\mu\right)  ds}}+(\gamma
+\mu)I(t)e^{^{\int_{0}^{t}\left(  \gamma+\mu\right)  ds}}\geq0.
\]
Then,\newline%
\[
\frac{d}{dt}\left(  I(t)e^{^{\int_{0}^{t}\left(  \gamma+\mu\right)  ds}%
}\right)  \geq0.
\]
Integrating both sides gives
\begin{equation}
I(t)\geq I(0)e^{-\int_{0}^{t}\left(  \gamma+\mu\right)  ds}>0.
\label{eqn:diffI}%
\end{equation}
Thus, it is clear from the solution that $I(t)$ is positive in this interval,
since the initial value $I(0)$ is positive and the exponential functions are
always positive.\newline\textbf{Positivity of }$S$: The proof follows the same
idea as the above proof. The proof is postponed to Appendix
\ref{app:nonNegativity}.\newline Thus, we have proved that $S(t)$ and $I(t)$
are positive for $0\leq t\leq\tau$, completing the base case.\newline
Induction step: Assume that $S(t)$ and $I(t)$ are positive for $m\tau\leq
t\leq m(\tau+1)$. \newline The positivity of $S(t)$ and $I(t)$ for
$m(\tau+1)\leq t\leq m(\tau+2)$ is proved analougously as in the base
case.\newline With the help of mathematical induction, the positivity for
$I(t)$ and $S(t)$ for all $t$ follows.\newline\textbf{Positivity of }$E$:
\newline Recall that $\operatorname*{supp}\Psi\subset\lbrack\theta,L]$, where
$0<\theta<L<\infty$. The dynamics of $E$ can thus be evaluated as\newline%
\begin{align*}
\frac{dE(t)}{dt}  &  =\beta(t)I(t)S(t)-\int_{\mathbb{R_{+}}}\beta
(t-\tau)I(t-\tau)S(t-\tau)\Psi(\tau)d\tau,\\
&  =\beta(t)I(t)S(t)-\int_{\theta}^{L}\beta(t-\tau)I(t-\tau)S(t-\tau)\Psi
(\tau)d\tau.
\end{align*}
Integrating both sides gives
\[
E(t)=E(0)+\int_{0}^{t}\int_{\theta}^{L}\left(  \beta(s)I(s)S(s)-\beta
(s-\tau)I(s-\tau)S(s-\tau)\right)  \Psi(\tau)d\tau ds.
\]
By the use of Fubini's theorem, the linearity of the integral and a change of
variables $s-\tau\curvearrowright s$, we get
\begin{align*}
E(t)  &  =E(0)+\int_{\theta}^{L}\Psi(\tau)\int_{0}^{t}\left(  \beta
(s)I(s)S(s)-\beta(s-\tau)I(s-\tau)S(s-\tau)\right)  dsd\tau,\\
&  =E(0)+\int_{\theta}^{L}\Psi(\tau)\left(  \int_{0}^{t}\beta
(s)I(s)S(s)ds-\int_{-\tau}^{t-\tau}\beta(s)I(s)S(s)ds\right)  d\tau,\\
&  =E(0)+\int_{\theta}^{L}\Psi(\tau)\left(  \int_{t-\tau}^{t}\beta
(s)I(s)S(s)ds-\int_{-\tau}^{0}\beta(s)I(s)S(s)ds\right)  d\tau.
\end{align*}
Since $S(t)$ and $I(t)$ and $\beta(t)$ are non-negative, the history data
given by Equation \eqref{eq:Es} gives%
\begin{align*}
E(t)  &  =\beta_{0}c_{I}c_{S}\int_{\theta}^{L}\Psi(\tau)\tau d\tau
+\int_{\theta}^{L}\Psi(\tau)\int_{t-\tau}^{t}\beta(s)I(s)S(s)dsd\tau\\
&\quad-\beta
_{0}c_{I}c_{S}\int_{\theta}^{L}\Psi(\tau)\tau d\tau,\\
&  =\int_{\theta}^{L}\Psi(\tau)\int_{t-\tau}^{t}\beta(s)I(s)S(s)dsd\tau\geq0.
\end{align*}
Thus $E(t)$ is non-negative for all $t$.\newline\textbf{Positivity of
}$R^{\text{T}}$ \textbf{and} $R^{\text{P}}$: The demonstration of the
positivity of $R^{\text{T}}$ follows the same underlying idea as the proof
establishing for the positivity of $E$. It is clear that the positivity of
$R^{\text{T}}$ is hold by the positivity of $I.$ The proof is postponed to
Appendix \ref{app:nonNegativity}.
\end{proof}

We now turn our attention to the boundedness of the solution.

\begin{theorem}
[Boundedness]Under the same assumptions as in Theorem \ref{Thm:Non-negativity}%
, the solution $\{S,E,I,R^{\text{T}},R^{\text{P}}\}$ of System
\ref{model:mainGenNoBirthOrNaturalDeath} are all bounded.
\end{theorem}

\begin{proof}
The total population size is
\begin{equation}
N(t)=S(t)+E(t)+I(t)+R^{\text{T}}(t)+R^{\text{P}}(t). \label{eq:total}%
\end{equation}
The rate of change of the total population is%
\[
\frac{dN(t)}{dt}=\frac{dS(t)}{dt}+\frac{dE(t)}{dt}+\frac{dI(t)}{dt}%
+\frac{dR^{\text{T}}(t)}{dt}+\frac{dR^{\text{P}}(t)}{dt}=-\mu I(t).
\]
where the last equality follows from System
\ref{model:mainGenNoBirthOrNaturalDeath}. Hence, the total population is
non-increasing and since we have proved that all functions $S,E,I,R^{\text{T}%
},$ and $R^{\text{P}}$ are non-negative, the result is proved.
\end{proof}

\subsection{Some results on continuity and differentaibility}

In this section, we investigate the continuity and differentiability of the
integral of $H(t-\rho)\Phi(\rho),$these results will be applied in Section
\ref{sec:smoothness}. The results in this section are stated in a more general
setting and may thus be used in other applications.

\begin{theorem}
\label{the:conAnddiff} Let $H\in L^{1}(\mathbb{R})$ and let $\Phi\in
L^{1}(\mathbb{R})$ be non-negative with compact support. Define%
\[
F(t)=\int H(t-\rho)\Phi(\rho)d\lambda(\rho).
\]

\begin{enumerate}
\item If $H$ is continous then $F$ is continuous, that is
\[
\lim_{s\rightarrow t}F(s)=F(t).
\]

\item If $H\in C^{(k)}$ then $F$ is differentiable, and
\[
F^{(k)}(t)=\int_{-\infty}^{\infty}\frac{d^{k}}{dt^{k}}H(t-\rho))\Phi(\rho
)d\mu(\rho).
\]

\end{enumerate}
\end{theorem}

\begin{proof}
Part 1.\newline Let $f^{t_{n}}(\rho)=H(t_{n}-\rho)\Phi(\rho)$, where
$t_{n}\rightarrow t$ as $n\rightarrow\infty$.\newline We know that
$H(t_{n}-\rho)$ is continuous function That is,%
\[
\lim_{n\rightarrow\infty}H(t_{n}-\rho)=H(t-\rho).
\]
\newline Let $K$ be a compact set such that $\{t_{n}\}\cup\operatorname*{supp}%
(\Phi)\cup\{t\}\subset K.$\newline Since $\Phi\in L^{1},$ it follows that
$I=\{x:\Phi(x)=\infty\}$ is a null set (i.e. it has measure zero). For any
$\rho\in\mathbb{R}\backslash I$ the sequence $f^{t_{n}}(\rho)\rightarrow
f^{t}(\rho)=H(t-\rho)\Phi(\rho)$, i.e., $f_{n}\rightarrow f$ almost
everywhere. We also have that $0<\max_{x\in K}|H(x)|=C<\infty$ since the
continuous function $H$ is bounded on any compact set. Using that $\Phi\in
L^{1}(\mathbb{R})$ is non-negative, it is obvious that $|f^{t_{n}}(\rho)|\leq
C\Phi(\rho)$ for all $n$. The hypothesis in the Lebesgue dominated convergence
theorem (LDCT) are fulfilled for $f^{t_{n}}(\rho)$. Thus%
\begin{align*}
F(t)  &  =\int_{-\infty}^{\infty}\underbrace{H(t-\rho)\Phi(\rho)}_{=f^{t}%
(\rho)}d\lambda(\rho),\\
&  =\int_{-\infty}^{\infty}\lim_{n\rightarrow\infty}H(t_{n}-\rho)\Phi
(\rho)d\lambda(\rho),\\
&  =\int_{-\infty}^{\infty}\lim_{n\rightarrow\infty}f^{t_{n}}(\rho
)d\lambda(\rho).
\end{align*}
By LDCT, we get%
\begin{align*}
F(t)  &  =\lim_{n\rightarrow\infty}\int_{-\infty}^{\infty}f^{t_{n}}%
(\rho)d\lambda(\rho),\\
&  =\lim_{n\rightarrow\infty}\int_{-\infty}^{\infty}H(t_{n}-\rho)\Phi
(\rho)d\lambda(\rho),\\
&  =\lim_{n\rightarrow\infty}F(t_{n}).
\end{align*}
Hence $F$ is continuous.\newline Part 2.\newline By definition of derivative,
we have%
\begin{equation}
\frac{df^{t}(\rho)}{dt}=\lim_{n\rightarrow\infty}\frac{f^{t_{n}}(\rho
)-f^{t}(\rho)}{t_{n}-t}=\lim_{n\rightarrow\infty}h_{n}(\rho).
\label{eq:diff_fRho}%
\end{equation}
\newline The above statement is well defined since $f^{t_{n}}(\rho
)=H(t_{n}-\rho)\Phi(\rho)$ is differentiable with respect to $t$ for almost
all $\rho$.\newline The mean value theorem gives,
\[
\left\vert \frac{H(t_{n}-\rho)-H(t-\rho)}{t_{n}-t}\right\vert =|H^{\prime
}(c_{n})|\leq\sup_{x\in K}|H^{\prime}(x)|=C<\infty.
\]
Then,%
\begin{align}
\left\vert h_{n}(\rho)\right\vert  &  =\left\vert \frac{f^{t_{n}}(\rho
)-f^{t}(\rho)}{t_{n}-t}\right\vert ,\nonumber\\
&  =\left\vert \frac{H(t_{n}-\rho)-H(t-\rho)}{t_{n}-t}\right\vert \Phi
(\rho),\nonumber\\
&  \leq C\Phi(\rho), \label{eq:bounded}%
\end{align}
\newline for almost every $\rho$ since $\Phi(\rho)$ is a bounded function for
each $\rho\in I$. The hypothesis in the dominated convergence theorem are
fulfilled for $h_{n}$.\newline By definition of derivative,%
\[
F^{^{\prime}}(t)=\lim_{n\rightarrow\infty}\frac{F(t_{n})-F(t)}{t_{n}-t}.
\]
\newline Substitute $F(t_{n})=\int f^{t_{n}}(\rho)d\lambda(\rho)$ in the above
equation, we get%
\begin{align*}
F^{^{\prime}}(t)  &  =\lim_{n\rightarrow\infty}\frac{\int f^{t_{n}}%
(\rho)d\lambda(\rho)-\int f^{t}(\rho)d\lambda(\rho)}{t_{n}-t},\\
&  =\lim_{n\rightarrow\infty}\int\frac{\left(  f^{t_{n}}(\rho)-f^{t}%
(\rho)\right)  }{t_{n}-t}d\lambda(\rho)=\lim_{n\rightarrow\infty}\int
h_{n}(\rho)d\lambda(\rho).
\end{align*}
By Lebesgue dominated convergence theorem,
\[
F^{^{\prime}}(t)=\int\lim_{n\rightarrow\infty}h_{n}(\rho)d\lambda(\rho
)=\int\lim_{n\rightarrow\infty}\frac{\left(  f^{t_{n}}(\rho)-f^{t}%
(\rho)\right)  }{t_{n}-t}d\lambda(\rho).
\]
By using Equation \eqref{eq:diff_fRho}, we get%
\begin{align*}
F^{^{\prime}}(t)  &  =\int\frac{df^{t}(\rho)}{dt}d\lambda(\rho),\\
&  =\int\frac{d}{dt}H(t-\rho)\Phi(\rho)d\lambda(\rho),
\end{align*}
for almost every $\rho.$ To prove for the higher derivative ,$F^{(k)}(t),$ we
can follow the same procedure.
\end{proof}

\begin{corollary}
\label{cor:continuous} Let $H:\mathbb{R}\rightarrow\mathbb{R}$ is continuous,
and let $\Phi\in L^{1}(\mathbb{R})$ be non-negative with compact support.
Define, for all $t\geq a$%
\[
F(t)=\int_{a}^{t}H(t-\rho)\Phi(\rho)d\lambda(\rho),
\]
then $F$ is continuous, that is
\[
\lim_{s\rightarrow t}F(s)=F(t).
\]

\end{corollary}

\begin{proof}
The continuity of $F$ is proved by using H\"{o}lder's inequality and being
uniformly continuous of $H$ The proof is postponed to Appendix
\ref{app:proofCor}.
\end{proof}

\subsection{Enhancing Solution Smoothness Over Time\label{sec:smoothness}}

The smoothness of the solution of System
\ref{model:mainGenNoBirthOrNaturalDeath} are established in this section with
the help of Theorem \ref{the:conAnddiff} and Corollary \ref{cor:continuous}.
The progression of solution smoothness is evident as the window interval is
systematically shifted to the right, highlighting a continual enhancement over time.

\begin{definition}
Let $r>0$. A function $f:\mathbb{R}_{+}\rightarrow\mathbb{R}_{+}$ is said to
be $r$-\emph{increasingly smooth} if $f\in C^{n+1}((rn,\infty))$ for all
$k=0,1,\ldots$ .
\end{definition}

\begin{theorem}
\label{the:smoothness}Let System~\ref{model:mainGenNoBirthOrNaturalDeath}
have constant history data given by equations \eqref{eq:HS} and \eqref{eq:HI}, and
initial conditions $E(0)$, $R^{\text{T}}(0)$ and $R^{\text{P}}(0)$ given by
equations \eqref{eq:Es}--\eqref{eq:RPs}. Then, there exist an $r>0$ such that the
solution functions to System~\ref{model:mainGenNoBirthOrNaturalDeath} are
$r$-\emph{increasingly smooth}.
\end{theorem}

\begin{proof}
Let $\epsilon,\theta$, $M$ and $L$ be defined by equations
\eqref{Constant:ThetaAndL} and \eqref{Constant:Epsilon}. We set $\varepsilon
=\min(\epsilon,\theta)$. We will demonstrate that $r=L$ holds true within the
context of the theorem's hypothesis. The proof is divided into the parts:
\newline i) -- v) Continuity on $(-\infty,\varepsilon)$ of $I$,$\ S$\ and
differentiability on $(-\infty,\varepsilon)$ of $R^{\text{T}},\ R^{\text{P}}$
and $E$, \newline vi-vii) Differentiability on $(0,2\varepsilon)$ of
$I$,$\ S,$ $R^{\text{T}},\ R^{\text{P}}$ and $E$,\newline viii)
Differentiability on $(0,\infty)$ of $I,\ S,\ R^{\text{T}},\ R^{\text{P}}$ and
$E$\newline ix) $L$-increasingly smooth $I,\ S,\ R^{\text{T}},\ R^{\text{P}}$
and $E$. \newline i) \textbf{Continuity on }$(-\infty,\varepsilon)$\textbf{ of
}$I$: Remember that%
\[
\frac{dI}{dt}=\int_{\mathbb{R_{+}}}\beta(t-\tau)I(t-\tau)S(t-\tau)\Psi
(\tau)d\tau-\gamma I(t)-\mu I(t).
\]
We define
\begin{equation}
g(t)=\int_{\varepsilon}^{L}\beta(t-\tau)I(t-\tau)S(t-\tau)\Psi(\tau)d\tau.
\label{eq:g(t)}%
\end{equation}
For $0<t<\varepsilon,$%
\[
g(t)=\int_{\varepsilon}^{L}\beta(t-\tau)c_{I}c_{S}\Psi(\tau)d\tau,
\]
From Corollary \ref{the:conAnddiff}, the function $g(t)$ is smooth and
non-negative. Thus%
\begin{equation}
\frac{dI}{dt}=g(t)-(\gamma+\mu)I(t),\nonumber
\end{equation}

From the above equation, we get%
\begin{equation}
I(t)=\left(  c_{I}+\int_{0}^{t}g(\zeta)e^{(\gamma+\mu)\zeta}d\zeta\right)
e^{-(\gamma+\mu)t}. \label{eq:I(t)}%
\end{equation}
Then, we get%
\[
\lim_{t\searrow0}I(t)=c_{I}.
\]
For $-\infty\leq t\leq0,$ $I(t)=c_{I},$ \newline Thus, $I(t)$ is the
continuous function for $-\infty<t<\varepsilon$.\newline Note that,%
\begin{equation}
\lim_{t\searrow0}I^{^{\prime}}(t)\neq\lim_{t\nearrow0}I^{^{\prime}}(t).
\label{eq:DiffI(0)}%
\end{equation}
We note that the solution $I$ to System
\ref{model:mainGenNoBirthOrNaturalDeath} satisfy $I\in C^{\infty
}((0,\varepsilon))\cap C((-\infty,\varepsilon))$.\newline ii)
\textbf{Continuity on }$(-\infty,\varepsilon)$\textbf{ of }$S$: Remember that%
\[
\frac{dS}{dt}=-\beta(t)I(t)S(t)+p\gamma\int_{\mathbb{R_{+}}}I(t-\rho)\Phi
(\rho)d\rho.
\]
\newline For $0<t<\varepsilon$, we evaluate%
\begin{align*}
&  p\gamma\int_{\mathbb{\varepsilon}}^{M}I(t-\rho)\Phi(\rho)d\rho,\\
&  =p\gamma c_{I}\int_{\varepsilon}^{M}\Psi(\rho)d\rho,\\
&  =p\gamma c_{I}.
\end{align*}
Thus,%
\begin{equation}
\frac{dS}{dt}=-\beta(t)I(t)S(t)+p\gamma c_{I}. \label{eq:diffSFI}%
\end{equation}
We rewrite the above equation, then%
\begin{align}
\left(  e^{\int_{0}^{t}\beta(s)I(s)ds}S(t)\right)  ^{\prime}  &  =p\gamma
c_{I}e^{\int_{0}^{t}\beta(s)I(s)ds},\\
S(t)  &  =\frac{1}{e^{\int_{0}^{t}\beta(s)I(s)ds}}\left(  S(0)+\int_{0}%
^{t}p\gamma c_{I}e^{\int_{0}^{\zeta}\beta(s)I(s)ds}d\zeta\right)  .
\label{eq:SFrom0ToEp}%
\end{align}
\newline Then, we get%
\begin{align*}
\lim_{t\searrow0}S(t)  &  =\lim_{t\searrow0}\left(  \frac{1}{e^{\int_{0}%
^{t}\beta(s)I(s)ds}}\left(  S(0)+\int_{0}^{t}p\gamma c_{I}e^{\int_{0}^{\zeta
}\beta(s)I(s)ds}d\zeta\right)  \right)  ,\\
&  =S(0)=c_{S}.
\end{align*}
For $-\infty\leq t\leq0,$ $S(t)=c_{S}.$\newline Then%
\begin{equation}
\lim_{t\searrow0}S(t)=c_{S}=\lim_{t\nearrow0}S(t). \label{eq:ContS0}%
\end{equation}
Thus $S$ is a continuous function for $-\infty<t<\varepsilon.$\newline Note
that%
\[
\lim_{t\nearrow0}S^{\prime}(t)\neq\lim_{t\searrow0}S^{\prime}(t).
\]
\newline We note that the solution $S$ to System
\ref{model:mainGenNoBirthOrNaturalDeath} satisfy $S\in C^{\infty
}((0,\varepsilon))\cap C((-\infty,\varepsilon))$.\newline iii)--v)
\textbf{Differentiability on }$(-\infty,\varepsilon)$\textbf{ of }%
$R^{\text{T}}$, $R^{\text{P}}$, and $E$: Recall that%
\begin{equation}
\frac{dR^{\text{T}}(t)}{dt}=p\gamma I(t)-p\gamma\int_{0}^{\infty}I(t-\rho
)\Phi(\rho)d\rho. \label{eq:dittRT}%
\end{equation}
For $-\infty<t<\varepsilon,$%
\[
\frac{dR^{\text{T}}(t)}{dt}=p\gamma(I(t)-c_{I}).
\]
Thus $R^{\text{T}}(t)\in C^{1}((-\infty,\varepsilon))$ since $I\in
C((-\infty,\varepsilon))$. \newline The same results, i.e., $R^{\text{T}%
}(t),E(t)\in C^{1}((-\infty,\varepsilon))$, are proved in an analougous
way.\newline vi) \textbf{Differentialbility on }$(0,2\varepsilon)$\textbf{ of}
$I$: Again, we recall%
\begin{equation}
\frac{dI}{dt}=\int_{\mathbb{R_{+}}}\beta(t-\tau)I(t-\tau)S(t-\tau)\Psi
(\tau)d\tau-\gamma I(t)-\mu I(t). \label{eq:IsmoothZeroToTwoEpsilon}%
\end{equation}
By using Equation \eqref{eq:g(t)}, we have%
\begin{align*}
g(t)=  &  \int_{\varepsilon}^{L}\beta(t-\tau)I(t-\tau)S(t-\tau)\Psi(\tau
)d\tau,\\
=  &  \underbrace{\int_{\varepsilon}^{2\varepsilon}\beta(t-\tau)I(t-\tau
)S(t-\tau)\Psi(\tau)d\tau}_{=g_{1}(t)}\\
&  +\underbrace{\int_{2\varepsilon}^{L}\beta(t-\tau)I(t-\tau)S(t-\tau
)\Psi(\tau)d\tau.}_{=g_{2}(t)}%
\end{align*}
\textbf{Remark}: We have assumed that $2\varepsilon<L$, if this is not the
case, one can set $g_{2}(t)\equiv0$ for the remainder of the proof, this. In
similar situations below, we adopt the same argument.\newline For
$\varepsilon<t<2\varepsilon,$%
\begin{align*}
g_{2}(t)  &  =\int_{2\varepsilon}^{L}\beta(t-\tau)I(t-\tau)S(t-\tau)\Psi
(\tau)d\tau,\\
&  =\int_{2\varepsilon}^{L}\beta(t-\tau)c_{I}c_{S}\Psi(\tau)d\tau,\\
&  =c_{I}c_{S}\int_{2\varepsilon}^{L}\beta(t-\tau)\Psi(\tau)d\tau.
\end{align*}
Furthermore,%
\begin{align*}
g_{1}(t)=  &  \int_{\varepsilon}^{2\varepsilon}\beta(t-\tau)I(t-\tau
)S(t-\tau)\Psi(\tau)d\tau,\\
=  &  \int_{\varepsilon}^{t}\beta(t-\tau)I(t-\tau)S(t-\tau)\Psi(\tau)d\tau\\
&  +\int_{t}^{2\varepsilon}\beta(t-\tau)I(t-\tau)S(t-\tau)\Psi(\tau)d\tau.
\end{align*}
By using the given history data, we get%
\begin{align*}
g_{1}(t)=  &  \int_{\varepsilon}^{t}\beta(t-\tau)I(t-\tau)S(t-\tau)\Psi
(\tau)d\tau+\int_{t}^{2\varepsilon}\beta(t-\tau)c_{I}c_{S}\Psi(\tau)d\tau,\\
=  &  \int_{\varepsilon}^{t}\beta(t-\tau)I(t-\tau)S(t-\tau)\Psi(\tau
)d\tau+c_{I}c_{S}\int_{t}^{2\varepsilon}\beta(t-\tau)\Psi(\tau)d\tau.
\end{align*}
By adding $g_{1}$ and $g_{2}$, it gives%
\begin{align*}
g(t)=  &  g_{1}(t)+g_{2}(t),\\
=  &  \int_{\varepsilon}^{t}\beta(t-\tau)I(t-\tau)S(t-\tau)\Psi(\tau
)d\tau+c_{I}c_{S}\int_{t}^{2\varepsilon}\beta(t-\tau)\Psi(\tau)d\tau\\
&  +c_{I}c_{S}\int_{2\varepsilon}^{L}\beta(t-\tau)\Psi(\tau)d\tau,\\
=  &  \int_{\varepsilon}^{t}\beta(t-\tau)I(t-\tau)S(t-\tau)\Psi(\tau
)d\tau+c_{I}c_{S}\int_{t}^{L}\beta(t-\tau)\Psi(\tau)d\tau.
\end{align*}
Substituting this into Equation \eqref{eq:IsmoothZeroToTwoEpsilon} gives%
\begin{equation}
\frac{dI}{dt}=\int_{\varepsilon}^{t}\beta(t-\tau)I(t-\tau)S(t-\tau)\Psi
(\tau)d\tau+c_{I}c_{S}\int_{t}^{L}\beta(t-\tau)\Psi(\tau)d\tau~-\gamma
I(t)-\mu I(t). \label{eq:diffISI}%
\end{equation}
Note that $t-\tau$ fall into the interval $0<t-\tau<\varepsilon$, hence by
using the Corollary \ref{cor:continuous}, the integrals on the right hand side
are continuous functions for all $t$ and a.e. $\tau$, because $S,I,\beta\in
C^{\infty}((0,\varepsilon))$ together with the assumptions on $\Psi$. The
remaining terms are obviously continuous, hence $I\in C^{1}((\varepsilon
,2\varepsilon))$.\newline Turning our attention to%
\begin{align*}
\lim_{t\searrow\varepsilon}I^{^{\prime}}(t)  &  =\lim_{t\searrow\varepsilon
}\Big(  \int_{\varepsilon}^{t}\beta(t-\tau)I(t-\tau)S(t-\tau)\Psi(\tau
)d\tau+c_{I}c_{S}\int_{t}^{L}\beta(t-\tau)\Psi(\tau)d\tau\Big.\\
&\qquad\Big.-\gamma I(t)-\mu
I(t)\Big)  ,\\
&  =c_{I}c_{S}\int_{\varepsilon}^{L}\beta(t-\tau)\Psi(\tau)d\tau-(\gamma
+\mu)I(\varepsilon).
\end{align*}
For $0<t<\varepsilon,$by using Equation \eqref{eq:IsmoothZeroToTwoEpsilon}, we
get%
\begin{align}
\lim_{t\nearrow\varepsilon}I^{^{\prime}}(t)  &  =\lim_{t\nearrow\varepsilon
}\left(  c_{I}c_{S}\int_{t}^{L}\beta(t-\tau)\Psi(\tau)d\tau-(\gamma
+\mu)I(t)\right) \label{eq:DiffIAtEp}\\
&  =c_{I}c_{S}\int_{\varepsilon}^{L}\beta(t-\tau)\Psi(\tau)d\tau-(\gamma
+\mu)I(\varepsilon)=\lim_{t\searrow\varepsilon}I^{^{\prime}}(t).
\end{align}
Thus $I$ $\in$ $C^{1}((0,2\varepsilon))$.\newline vii)
\textbf{Differentialbility on }$(0,2\varepsilon)$\textbf{ of }$S$:%
\[
\frac{dS}{dt}=-\beta(t)I(t)S(t)+p\gamma\int_{\mathbb{R_{+}}}I(t-\rho)\Phi
(\rho)d\rho.
\]
For $\varepsilon<t<2\varepsilon,$%
\begin{align*}
\frac{dS}{dt}  &  =-\beta(t)I(t)S(t)+p\gamma\int_{\mathbb{\varepsilon}}%
^{M}I(t-\rho)\Phi(\rho)d\rho,\\
&  =-\beta(t)I(t)S(t)+p\gamma\int_{\mathbb{\varepsilon}}^{t}I(t-\rho)\Phi
(\rho)d\rho+p\gamma\int_{t}^{M}I(t-\rho)\Phi(\rho)d\rho,\\
&  =-\beta(t)I(t)S(t)+p\gamma\int_{\mathbb{\varepsilon}}^{t}I(t-\rho)\Phi
(\rho)d\rho+p\gamma c_{I}\int_{t}^{M}\Phi(\rho)d\rho.
\end{align*}
\newline Then,%
\begin{align}
\lim_{t\searrow\varepsilon}S^{^{\prime}}(t)  &  =\lim_{t\searrow\varepsilon
}\left(  -\beta(t)I(t)S(t)+p\gamma\int_{\mathbb{\varepsilon}}^{t}I(t-\rho
)\Phi(\rho)d\rho+p\gamma c_{I}\int_{t}^{M}\Phi(\rho)d\rho\right)  ,\nonumber\\
&  =-\beta(\varepsilon)I(\varepsilon)S(\varepsilon)+p\gamma c_{I}%
\int_{\varepsilon}^{M}\Phi(\rho)d\rho,\nonumber\\
&  =-\beta(\varepsilon)I(\varepsilon)S(\varepsilon)+p\gamma c_{I}.
\label{eq:DiffSAtEp}%
\end{align}
For $0<t<\varepsilon,$ by using the Equation \eqref{eq:diffSFI}, we get%
\begin{align*}
\lim_{t\nearrow\varepsilon}S^{^{\prime}}(t)  &  =\lim_{t\nearrow\varepsilon
}\left(  -\beta(t)I(t)S(t)+p\gamma c_{I}\right)  ,\\
&  =-\beta(\varepsilon)I(\varepsilon)S(\varepsilon)+p\gamma c_{I}%
=\lim_{t\searrow\varepsilon}S^{^{\prime}}(t).
\end{align*}
Thus $S\in C^{1}((0,2\varepsilon))$.\newline We conclude that $R^{T}\in
C^{2}((0,2\varepsilon)),R^{P}\in C^{2}((0,2\varepsilon))$ and $E\in
C^{2}((0,2\varepsilon))$ since $S\in C^{1}((0,2\varepsilon)).$and $I\in
C^{1}((0,2\varepsilon)).$\newline viii) \textbf{Differentiability on
}$(0,\infty)$\textbf{ of }$I,S,R^{\text{T}},R^{\text{P}}$\textbf{ and }%
$E$:\newline We now prove that $I\in C^{1}((0,\infty))$ by using mathematical
induction. First, we have proved that $S,I\in C^{1}((0,2\varepsilon)$. We
assume that $S,~I\in C^{1}((0,(k+1)\varepsilon))$ is true, for a fixed integer
$k\geq1$. We will now prove that $I\in C^{1}((0,(k+2)\varepsilon))$,is true,
that is, $I^{^{\prime}}\in C((0,(k+2)\varepsilon)).$ \newline We will have two
case which are $\varepsilon<t<L$ and $t>L$ for $0<t<(k+2)\varepsilon.$\newline
For $\varepsilon<t<L<(k+2)\varepsilon,$\newline From Equation
\eqref{eq:Idynamic}, we have
\begin{align*}
\frac{dI}{dt}=  &  \int_{\mathbb{\varepsilon}}^{L}\beta(t-\tau)I(t-\tau
)S(t-\tau)\Psi(\tau)d\tau-\gamma I(t)-\mu I(t),\\
=  &  \int_{\mathbb{\varepsilon}}^{t}\beta(t-\tau)I(t-\tau)S(t-\tau)\Psi
(\tau)d\tau\\
&  +\int_{t}^{L}\beta(t-\tau)I(t-\tau)S(t-\tau)\Psi(\tau)d\tau-(\gamma
+\mu)I(t).
\end{align*}
By using the given history data, we get%
\[
\frac{dI}{dt}=\int_{\mathbb{\varepsilon}}^{t}\beta(t-\tau)I(t-\tau
)S(t-\tau)\Psi(\tau)d\tau+\int_{t}^{L}\beta(t-\tau)c_{I}c_{S}\Psi(\tau
)d\tau-(\gamma+\mu)I(t).
\]
\newline Since $\varepsilon-\tau<t-\tau<(k+2)\varepsilon-\varepsilon
=(k+1)\varepsilon,$ the continuity of $I^{^{\prime}}$ follows from assumption
for all $0<t<(k+2)\varepsilon$ and a.e. $\tau$ by using the Corollary
\ref{cor:continuous}.\newline For $L<t<(k+2)\varepsilon,$\newline From
Equation \eqref{eq:Idynamic}, we have
\begin{equation}
\frac{dI}{dt}=\int_{\mathbb{\varepsilon}}^{L}\beta(t-\tau)I(t-\tau
)S(t-\tau)\Psi(\tau)d\tau-\gamma I(t)-\mu I(t). \label{eq:derivative_I}%
\end{equation}
Since $t-\tau\leq t-\varepsilon<(k+2)\varepsilon-\varepsilon=(k+1)\varepsilon
$, the continuity of $I^{^{\prime}}$ follows from the assumption for all
$L<t<(k+2)\varepsilon$ and a.e. $\tau$ by applying the Corollary
\ref{cor:continuous}. Therefore $I\in C^{1}((0,(k+2)\varepsilon))$.\newline
Now let's prove that $S\in C^{1}((0,(k+2)\varepsilon))$, that is $S^{^{\prime
}}\in C((0,(k+2)\varepsilon)).$ \newline Let $0<t<(k+2)\varepsilon$ and let
$g(t)=p\gamma\int_{\mathbb{R_{+}}}I(t-\rho)\Phi(\rho)d\rho$. From above, we
have $I\in C^{1}((0,(k+2)\varepsilon)).$\newline Now\newline%
\begin{equation}
\frac{dS}{dt}=-\beta(t)I(t)S(t)+g(t). \label{eq:diffOfS}%
\end{equation}
Therefore $S\in C^{1}((0,(k+2)\varepsilon))$ since $S$ depend on only $I$ by
multiplying with the integration factor for both sides of above
equations.\newline We have proved that $S,I\in C^{1}((0,(k+2)\varepsilon)).$
By mathematical induction, we get $S,I\in C^{1}(0,\infty).$\newline We
conclude that $R^{T},R^{P}$ and $E$ are in $C^{2}((0,\infty))$ since they only
depend on $S$ and $I$.\newline ix) $L$\textbf{-increasingly smooth}
$I,S,R^{\text{T}},R^{\text{P}}$\textbf{ and }$E$:\newline The progressive
smoothness of $I$ will be proved below, again using mathematical induction. As
a first step, we have proved that $S,I\in C^{1}((0,\infty)).$ We assume for
$n=k$ that $S,I\in C^{k}(((k-1)L,\infty))$ is true, and proceed below by
proving that $I\in C^{k+1}((kL,\infty))$.\newline Assume $t>kL$. We have
$S,I\in C^{k}(((k-1)L,\infty))$ by assumption. Differentiating Equation
\eqref{eq:derivative_I} gives%
\[
I^{(k+1)}(t)=\frac{d^{(k)}}{dt^{(k)}}\left(  \int_{\mathbb{\varepsilon}}%
^{L}\beta(t-\tau)I(t-\tau)S(t-\tau)\Psi(\tau)d\tau-\gamma I(t)-\mu
I(t)\right)  .
\]
\newline For a.e. $\tau$, the k$^{th}$ derivative w.r.t. $t$ of $\beta
(t-\tau)I(t-\tau)S(t-\tau)\Psi(\tau)$ is continuous for any $t>kL$ since
$t-\tau\geq t-L>kL-L=(k-1)L,$ that is, $t-\tau>(k-1)L$. Thus, by Theorem
\ref{the:conAnddiff}, we get\newline%
\begin{equation}
I^{(k+1)}(t)=\left(  \int_{\mathbb{\varepsilon}}^{L}\frac{d^{k}}{dt^{k}%
}\left(  \beta(t-\tau)I(t-\tau)S(t-\tau)\right)  \Psi(\tau)d\tau-(\gamma
+\mu)I^{(k)}(t)\right)  . \label{eq:diffkOfI}%
\end{equation}
\newline It follows that $I^{k+1}(t)$ is continuous when $t>kL$ since the
integrand of the Equation \eqref{eq:diffkOfI} is a continuous function for
almost everywhere $\tau\in\lbrack\varepsilon,L]$ and all $t>kL$ by Theorem
\ref{the:conAnddiff}.\newline That is $I\in C^{k+1}((kL,\infty))$. Thus, $I\in
C^{n+1}((nL,\infty))$ is valid for any $n=0,1,...$ by mathematical
induction.\newline By definition, $I$ is $L$-\emph{increasingly smooth}%
.\newline Similarly we prove that $S$ is $L$-\emph{increasingly smooth} since
$S$ depends on only $I$ by multiplying with the integration factor for both
sides of Equation \eqref{eq:diffOfS}.\newline According to System
\ref{model:mainGenNoBirthOrNaturalDeath}, $E,$ $R^{T},R^{P}$ and $D$ are also
$L$-\emph{increasingly smooth }since they only depend on $S$ and $I$.
\end{proof}

\section{Numerical simulation\label{sec:numericalSimulation}}

In this section, we introduce a discrete time delay endemic version of System
\ref{model:mainGenNoBirthOrNaturalDeath}. An investigation of how the solution
of this discrete time delay model approximate the exact solution of System
\ref{model:mainGenNoBirthOrNaturalDeath}. We investigate the solution
behaviour for \emph{single, few }and\emph{ multiple discrete time delays}. The
exact solution and numerical approximation algorithms are described below.

\begin{table}[h]
\centering%
\begin{tabular}
[c]{|llll|}\hline
Symbol & Value & Unit & Interpretation\\\hline
$\gamma$ & 0.1$(1/\hat{\gamma})$ & $\text{day}^{-1}$ & recovery rate
(1/duration of sickness)\\\hline
$\beta$ & 0.5/N & $\text{day}^{-1}$ & contact rate\\\hline
$I_{FR}$ & 0.475 & $\text{day}^{-1}$ & infection fatality risk\\\hline
$p$ & 0.9 & $-$ & proportional immunity parameter\\\hline
$\tau$ & 5-15 & $\text{day}^{-1}$ & latent time\\\hline
$\rho$ & 200-250 & $\text{day}^{-1}$ & duration of temporary immunity\\\hline
\end{tabular}
\caption{Description of model parameters. The values fall within the interval
for Ebola, see Table 2 in \cite{disease2024}.}%
\label{tableVar}%
\end{table}

\subsection{Discrete lag endemic models}

The endemic discrete time delay model is defined as%
\begin{align}
S^{\prime}(t)  &  =-\beta(t)I(t)S(t)+p\gamma\sum_{i=1}^{N_{\rho}}\omega
_{i}^{\rho}I(t-\rho_{i}),\nonumber\\
E^{\prime}(t)  &  =\beta(t)I(t)S(t)-\sum_{j=1}^{N_{\tau}}\omega_{j}^{\tau
}\beta(t-\tau_{j})I(t-\tau_{j})S(t-\tau_{j}),\nonumber\\
I^{\prime}(t)  &  =\sum_{j=1}^{N_{\tau}}\omega_{j}^{\tau}\beta(t-\tau
_{j})I(t-\tau_{j})S(t-\tau_{j})-\gamma I(t)-\mu I(t),\label{system:Discrete}\\
R^{\text{T}^{\prime}}(t)  &  =p\gamma I(t)-p\gamma\sum_{i=1}^{N_{\rho}}%
\omega_{i}^{\rho}I(t-\rho_{i}),\nonumber\\
R^{\text{P}^{\prime}}(t)  &  =(1-p)\gamma I(t),\nonumber\\
D^{\prime}(t)  &  =\mu I(t).\nonumber
\end{align}
where the non-negative weights satisfies $\sum_{j=1}^{N_{\tau}}\omega
_{j}^{\tau}=\sum_{i=1}^{N_{\rho}}\omega_{i}^{\rho}=1$. In this article, we
will call the above system as the \emph{discrete }$(N_{\tau},N_{\rho})$\emph{
model}. This above system has been studied for the case $N_{\rho}=1$ and
$N_{\tau}=1$ in \cite{disease2024}. An idea of how this discrete $(N_{\tau
},N_{\rho})$ model may be viewed as discretization of System
\ref{model:mainGenNoBirthOrNaturalDeath} is outlined as follows: Choose
$\omega_{i}^{\rho}$ and $\omega_{j}^{\tau}$ appropriately, such that%
\[
\sum_{i=1}^{N_{\rho}}\omega_{i}^{\rho}I(t-\rho_{i})\underset{N_{\rho
}\rightarrow\infty}{\longmapsto}\int_{\mathbb{R_{+}}}I(t-\rho)\Phi
(\rho)d\lambda(\rho),
\]
and%
\[
\sum_{j=1}^{N_{\tau}}\omega_{j}^{\tau}\beta(t-\tau_{j})I(t-\tau_{j}%
)S(t-\tau_{j})\underset{N_{\rho
}\rightarrow\infty}\longmapsto\int_{\mathbb{R_{+}}}\beta(t-\tau)I(t-\tau
)S(t-\tau)\Psi(\tau)d\lambda(\tau).
\]
Moreover, the above sum can be viewed as the integrals using sums of dirac
measures, i.e.,%
\[
\sum_{i=1}^{N_{\rho}}\omega_{i}^{\rho}I(t-\rho_{i})=\int_{\mathbb{R_{+}}%
}I(t-\rho)d\nu(\rho),
\]
where $d\nu(\rho)=\sum_{i=1}^{N}\omega_{i}^{\rho}\delta(\rho-\rho_{i})$, here
$\delta$ is the Dirac measure concentrated at $\rho-\rho_{i}$.\newline To
motivate these arguments, we continue with a comparison of numerical solution
to discrete $(N_{\tau},N_{\rho})$ models (see System \ref{system:Discrete})
and System \ref{model:mainGenNoBirthOrNaturalDeath}.

\subsection{Kernel functions in simulations\label{sec:kernelAnalysis}}

In this chapter the continuous integral kernels of System
\ref{model:mainGenNoBirthOrNaturalDeath} and the discrete kernels of System
\ref{system:Discrete} are defined for utilization in the simulations.
Specifically, we employ the two triangle shaped kernels as follows:%
\begin{equation}
\Phi(\rho)=\left\{
\begin{array}
[c]{cc}%
\frac{1}{25^{2}}(\rho-200) & 200<\rho\leq225,\\
\frac{1}{25^{2}}(250-\rho) & 225<\rho<250,\\
0 & \text{elsewhere.}%
\end{array}
\right.  \label{eq:conKernelRho}%
\end{equation}
and
\begin{equation}
\Psi(\tau)=\left\{
\begin{array}
[c]{cc}%
\frac{1}{25}(\tau-5) & 5<\tau\leq10,\\
\frac{1}{25}(15-\tau) & 10<\tau<15,\\
0 & \text{elsewhere.}%
\end{array}
\right.  \label{eq:conKernelTau}%
\end{equation}
That is, the latency time is in the range $5\ $to $15$ days and the temporary
time of immunity is in the range $200\ $to $250$ days. These kernel functions
fall within the parameter values for Ebola, see Table 2 in \cite{disease2024}%
.\newline In the discrete model, System \ref{system:Discrete} is simulated
with three different settings of discrete time delays as follows:

\begin{enumerate}
\item \emph{Discrete }$(1,1)$ \emph{model (Single lag)}: With $N_{\tau
}=N_{\rho}=1$, the single time delays are $\tau=10$ and $\rho=225$, and
$\omega_{1}^{\tau}=\omega_{1}^{\rho}=1$.

\item \emph{Discrete }$(3,3)$ \emph{model (Few lags)}: With $N_{\tau}=N_{\rho
}=3$, resulting in few time delays set to be $\tau_{j}=5j$ and $\rho
_{j}=190+10j$, and $\omega^{\tau}=\omega^{\rho}=\frac{1}{4}(1,2,1)$, where
$j=1,2,3$.

\item \emph{Discrete }$(60,60)$ \emph{model (Multiple lags)}: With $N_{\tau
}=N_{\rho}=60$, resulting in multiple time delays set to be $\tau_{j}%
=5+j\cdot10/59$ and $\rho_{j}=200+j\cdot50/59$, with $\omega^{\tau}%
=\omega^{\rho}=\frac{1}{930}(1,2,...,29,30,30,29,...,2,1)$, where
$j=0,1,...,59$.
\end{enumerate}

\subsection{Exact solution of System 1\label{sec:exactSol}}

Investigating the proof theorem \ref{the:smoothness} of increasingly smooth
solution to System \ref{model:mainGenNoBirthOrNaturalDeath} an algorithm is
devised to find the exact solution on consecutive time windows. In this
chapter we write out this algorithm to suit our numerical example in which we
use ODE45 in matlab for a numerical solution. We divide System
\ref{model:mainGenNoBirthOrNaturalDeath}, into one system of ODEs and two well
defined integrals, as%

\begin{align}
S^{\prime}(t)  &  =-\beta(t)I(t)S(t)+p\gamma g(t),\nonumber\\
E^{\prime}(t)  &  =\beta(t)I(t)S(t)-h(t),\nonumber\\
I^{\prime}(t)  &  =h(t)-\gamma I(t)-\mu I(t),\label{eq:system1WithKernels}\\
R^{\text{T}^{\prime}}(t)  &  =p\gamma I(t)-p\gamma g(t),\nonumber\\
R^{\text{P}^{\prime}}(t)  &  =(1-p)\gamma I(t),\nonumber\\
D^{\prime}(t)  &  =\mu I(t),\nonumber
\end{align}
where
\begin{align}
g(t)  &  =\int_{200}^{250}I(t-\rho)\Phi(\rho)d\rho,\nonumber\\
&  =\frac{1}{25^{2}}\int_{200}^{225}I(t-\rho)(\rho-200)d\rho+\frac{1}{25^{2}%
}\int_{200}^{225}I(t-\rho)(250-\rho)d\rho, \label{eqn:Integral_g}%
\end{align}
and
\begin{align}
h(t)  &  =\int_{5}^{15}\beta(t-\tau)I(t-\tau)S(t-\tau)\Psi(\tau)d\tau
,\nonumber\\
&  =\frac{1}{25}\int_{5}^{10}\beta(t-\tau)I(t-\tau)S(t-\tau)(\tau
-5)d\tau\nonumber\\
&  +\frac{1}{25}\int_{10}^{15}\beta(t-\tau)I(t-\tau)S(t-\tau)(15-\tau)d\tau.
\label{eqn:Integral_h}%
\end{align}
The algorithm to find the exact solution is as follows:

\begin{enumerate}
\item Set $k=0$ for the initial condition

\item Solve for $g$ and $h$ for $k\leq t\leq k+5$ by using the history data of
$S$ and $I$

\item Solve System \ref{model:mainGenNoBirthOrNaturalDeath} by using the
information from step 1 and 2 for $k=k+5.$

\item goto 2)
\end{enumerate}

In Section \ref{sec:simulation}, this algorithm is used to find a numerical
solution of System \ref{eq:system1WithKernels} by the help of the ODE45 solver
in MATLAB in step 3. This solution is then compared with the numerical
solution of System \ref{system:Discrete} of the \emph{discrete }$(N_{\tau
},N_{\rho})$\emph{ models.} The discrete models, given by System
\ref{system:Discrete} are solved using the DDE23 solver in MATLAB.

\subsection{Simulation results\label{sec:simulation}}

In this section, we compare the solution of System \ref{eq:system1WithKernels}
and System \ref{system:Discrete}. The parameter values are given in Table
\ref{tableVar} and the kernels are given by \eqref{eq:conKernelRho} and
\eqref{eq:conKernelTau}. The initial total population is set to be 10 million
and the number of initially infected individuals is assigned to be 10. The
initial number of permanent recovered individuals and the initial disease
deaths are set to zero. Additionally, the equation \eqref{eq:Es}--\eqref{eq:RPs}
are used to determine the initial number of individuals in exposed stage,
temporary and permanent recovery stages, respectively. With the help of these
details, the initial susceptible individuals is calculated by using Equation
\eqref{eq:total}. The disease death rate, $\mu$, is calculated by%
\[
\mu=\frac{\gamma I_{\text{FR}}}{1-I_{\text{FR}}},
\]
where the detailed proof can be found in \cite{disease2024}.\newline In the
following illustrations, we simulate four distinct types of compartmental
population detailed in Section \ref{sec:kernelAnalysis} over a period of 10
years. Across all figures, a clear observation emerges---namely, that the
solution of the \emph{discrete }$(N_{\tau},N_{\rho})$\emph{ models} aligns
better and better with the solution of the continuous model when the number of
delays is increased.%

\begin{figure}[H]
\centering
\includegraphics[width=5.8cm, height=5cm]{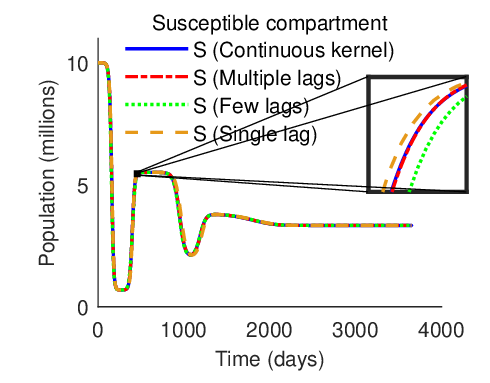}
\caption{The number of susceptible individuals is represented for four different models
under 10 years. The blue solid curve uses continuous time delay kernel
functions. The red dotted curve is simulated utilizing discrete (60,60) model,
the green dotted curve is plotted with discrete (3,3) model, and the yellow
dashed curve is produced with the discrete (1,1) model.}\label{fig:susceptible}
\end{figure}

\begin{figure}[H]
\centering
\includegraphics[width=5.8cm, height=5cm]{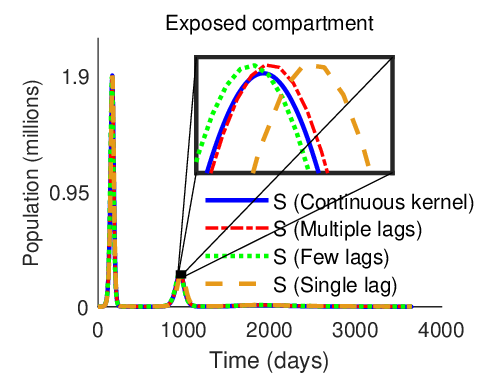}
\caption{The number of exposed individuals is represented for four different models
under 10 years. The blue solid curve uses continuous time delay kernel
functions. The red dotted curve is simulated utilizing discrete (60,60) model,
the green dotted curve is plotted with discrete (3,3) model, and the yellow
dashed curve is produced with the discrete (1,1) model.}\label{fig:exposed}
\end{figure}

\begin{figure}[H]
\centering
\includegraphics[width=5.8cm, height=5cm]{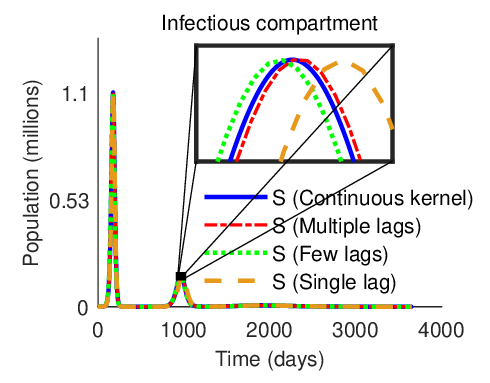}
\caption{
The number of infectious individuals is represented for four different models
under 10 years. The blue solid curve uses continuous time delay kernel
functions. The red dotted curve is simulated utilizing discrete (60,60) model,
the green dotted curve is plotted with discrete (3,3) model, and the yellow
dashed curve is produced with the discrete (1,1) model.}\label{fig:infected}
\end{figure}

\begin{figure}[H]
\centering
\includegraphics[width=5.8cm, height=5cm]{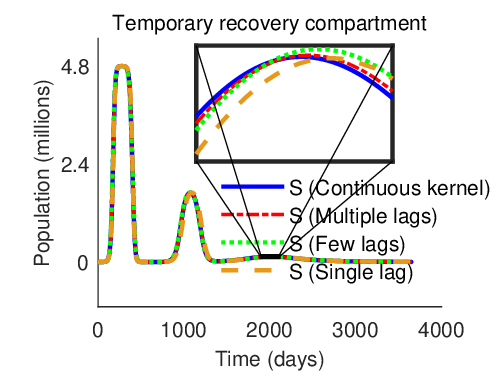}
\caption{The number of temporary recovery individuals is represented for four different
models under 10 years. The blue solid curve uses continuous time delay kernel
functions. The red dotted curve is simulated utilizing discrete (60,60) model,
the green dotted curve is plotted with discrete (3,3) model, and the yellow
dashed curve is produced with the discrete (1,1) model.
}\label{fig:temporary}
\end{figure}

\begin{figure}[H]
\centering
\includegraphics[width=5.8cm, height=5cm]{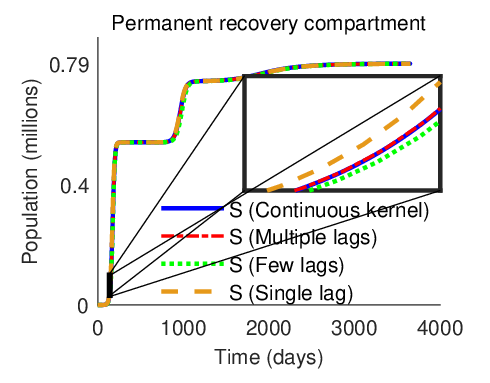}
\caption{The number of permanent recovery individuals is represented for four different
models under 10 years. The blue solid curve uses continuous time delay kernel
functions. The red dotted curve is simulated utilizing discrete (60,60) model,
the green dotted curve is plotted with discrete (3,3) model, and the yellow
dashed curve is produced with the discrete (1,1) model.}\label{fig:permanent}
\end{figure}

\begin{figure}[H]
\centering
\includegraphics[width=5.8cm, height=5cm]{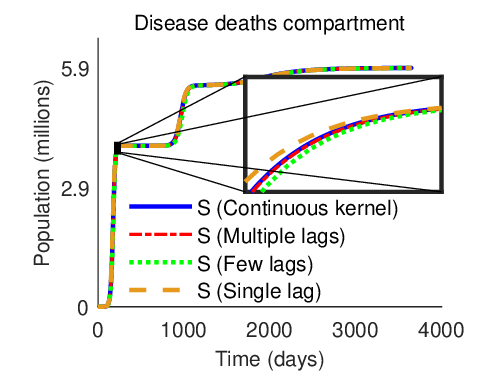}
\caption{The number of disease death individuals is represented for four different
models under 10 years. The blue solid curve uses continuous time delay kernel
functions. The red dotted curve is simulated utilizing discrete (60,60) model,
the green dotted curve is plotted with discrete (3,3) model, and the yellow
dashed curve is produced with the discrete (1,1) model.}\label{fig:dead}
\end{figure}

\section{Final remarks}

We propose an endemic model to describe the dynamics of an infectious disease,
incorporated with $L^{1}$ probability density delay kernel functions for the
latency and time lag of temporary immunity. The model is formulated as a
system of delay differential equations, and we ensure that System
\ref{model:mainGenNoBirthOrNaturalDeath} is well-posed, meaning that its
solution is non-negative and bounded. The solutions are also proved to possess
the property of being increasingly smooth over time.\newline The endemic
discrete time lag model is also proposed to approximate the exact solution of
System \ref{eq:system1WithKernels}. The simulation results indicate that the
\emph{discrete }$(N_{\tau},N_{\rho})$\emph{ model, }System
\ref{system:Discrete} converges to the solution of System
\ref{eq:system1WithKernels}.

In this paper, we discuss the possibility of viewing the discrete models using
kernels as integrals over Dirac measures. A future work will be designated to
a concrete proof that the numerical solutions of the discrete models, System
\ref{system:Discrete} converges to the exact solution of System
\ref{eq:system1WithKernels}.

\section{Acknowledgements}

This research was supported by Scholars at Risk and M\"{a}lardalen University
with funding from M\"{a}lardalen University and Formas (Dnr: GU 2022/1417).

\bibliographystyle{plain}
\bibliography{mybib}

\begin{thebibliography}{16}
\providecommand{\natexlab}[1]{#1}
\providecommand{\url}[1]{\texttt{#1}}
\expandafter\ifx\csname urlstyle\endcsname\relax
  \providecommand{\doi}[1]{doi: #1}\else
  \providecommand{\doi}{doi: \begingroup \urlstyle{rm}\Url}\fi

\bibitem[AK et~al.(2013)AK, SN, AL, PK, and P]{misra2013mathematical}
Misra AK, Mishra SN, Pathak AL, Srivastava PK, and Chandra P.
\newblock A mathematical model for the control of carrier-dependent infectious
  diseases with direct transmission and time delay.
\newblock \emph{Chaos, {S}olitons \& {F}ractals}, 57:\penalty0 41--53, 2013.

\bibitem[B et~al.(2020)B, MM, and A]{shayak2020temporary}
Shayak B, Sharma MM, and Misra A.
\newblock Temporary immunity and multiple waves of {COVID}-19.
\newblock \emph{med{R}xiv}, pages 2020--07, 2020.

\bibitem[C and L(2008)]{wei2008delayed}
Wei C and Chen L.
\newblock A delayed epidemic model with pulse vaccination.
\newblock \emph{Discrete {D}ynamics in {N}ature and {S}ociety}, 2008, 2008.

\bibitem[CB et~al.(1991)CB, EE, CE, and KC]{hall1991immunity}
Hall CB, Walsh EE, Long CE, and Schnabel KC.
\newblock Immunity to and frequency of reinfection with respiratory syncytial
  virus.
\newblock \emph{Journal of {I}nfectious {D}iseases}, 163\penalty0 (4):\penalty0
  693--698, 1991.

\bibitem[E and Y(1995)]{beretta1995global}
Beretta E and Takeuchi Y.
\newblock Global stability of an {SIR} epidemic model with time delays.
\newblock \emph{Journal of {M}athematical {B}iology}, 33\penalty0 (3):\penalty0
  250--260, 1995.

\bibitem[H et~al.(2022)H, Y, P, Z, EHY, Y, L, BJ, TK, and Z]{xin2022estimating}
Xin H, Li~Y, Wu~P, Li~Z, Lau EHY, Qin Y, Wang L, Cowling BJ, Tsang TK, and
  Li~Z.
\newblock Estimating the latent period of coronavirus disease 2019
  ({COVID}-19).
\newblock \emph{Clinical {I}nfectious {D}iseases}, 74\penalty0 (9):\penalty0
  1678--1681, 2022.

\bibitem[HI et~al.(2022)HI, L, MAY, M, NC, KH, W, Y, SSM, and
  MH]{abufares2022covid}
Abufares HI, Oyoun~Alsoud L, Alqudah MAY, Shara M, Soares NC, Alzoubi KH,
  El-Huneidi W, Bustanji Y, Soliman SSM, and Semreen MH.
\newblock {COVID}-19 vaccines, effectiveness, and immune responses.
\newblock \emph{International {J}ournal of {M}olecular {S}ciences}, 23\penalty0
  (23):\penalty0 15415, 2022.

\bibitem[I and Y(2016)]{al2016time}
Al-Darabsah I and Yuan Y.
\newblock A time-delayed epidemic model for {E}bola disease transmission.
\newblock \emph{Applied {M}athematics and {C}omputation}, 290:\penalty0
  307--325, 2016.

\bibitem[J and P(2006)]{arino2006time}
Arino J and Van Den~Driessche P.
\newblock Time delays in epidemic models.
\newblock \emph{Delay {D}ifferential {E}quations and {A}pplications, {NATO}
  {S}cience {S}eries}, 205:\penalty0 539--578, 2006.

\bibitem[JZ et~al.(2008)JZ, Z, QX, and ZY]{zhang2008analysis}
Zhang JZ, Jin Z, Liu QX, and Zhang ZY.
\newblock Analysis of a delayed {SIR} model with nonlinear incidence rate.
\newblock \emph{Discrete {D}ynamics in {N}ature and {S}ociety}, 2008, 2008.

\bibitem[L et~al.(2020)L, X, H, S, Y, and Y]{zhu2020dynamics}
Zhu L, Wang X, Zhang H, Shen S, Li~Y, and Zhou Y.
\newblock Dynamics analysis and optimal control strategy for a {SIRS} epidemic
  model with two discrete time delays.
\newblock \emph{Physica {S}cripta}, 95\penalty0 (3):\penalty0 035213, 2020.

\bibitem[M et~al.(2015)M, A, M, and N]{moosazadeh2015incidence}
Moosazadeh M, Bahrampour A, Nasehi M, and Khanjani N.
\newblock The incidence of recurrence of {T}uberculosis and its related factors
  in smear-positive pulmonary {T}uberculosis patients in {I}ran: {A}
  retrospective cohort study.
\newblock \emph{Lung {I}ndia: {O}fficial organ of {I}ndian {C}hest {S}ociety},
  32\penalty0 (6):\penalty0 557, 2015.

\bibitem[MA et~al.(2018)MA, PH, and L]{behr2018revisiting}
Behr MA, Edelstein PH, and Ramakrishnan L.
\newblock Revisiting the timetable of {T}uberculosis.
\newblock \emph{Bmj}, 362, 2018.

\bibitem[N et~al.(2022)N, E, and A]{guglielmi2022delay}
Guglielmi N, Iacomini E, and Viguerie A.
\newblock Delay differential equations for the spatially resolved simulation of
  epidemics with specific application to {COVID}-19.
\newblock \emph{Mathematical {M}ethods in the {A}pplied {S}ciences},
  45\penalty0 (8):\penalty0 4752--4771, 2022.

\bibitem[TN et~al.(2025)TN, L, and \AA]{disease2024}
Aye TN, Carlsson L, and Br\"{a}nnstr\"{o}m \AA.
\newblock On epidemic waves and fatalities in infectious diseases with latency
  and partial immunity.
\newblock \emph{Manuscript submitted for publication to Acta Biotheoretica},
  2025.

\bibitem[TRJ et~al.(2005)TRJ, MA, OM, VM, RW, and KM]{duintjer2005dynamic}
Duintjer TRJ, Pallansch MA, Kew OM, C{\'a}ceres VM, Sutter RW, and Thompson KM.
\newblock A dynamic model of poliomyelitis outbreaks: {L}earning from the past
  to help inform the future.
\newblock \emph{American {J}ournal of {E}pidemiology}, 162\penalty0
  (4):\penalty0 358--372, 2005.

\end{thebibliography}

\section{Appendix}

\subsection{Proof of Theorem \ref{Thm:Non-negativity}\label{app:nonNegativity}%
}

The following proof shows that the non-negativity of $S,$ $R^{\text{T}}$ and
$R^{\text{P}}.$\newline\textbf{Positivity of }$S$:\newline Consider $\frac
{dS}{dt}=-\beta(t)I(t)S(t)+p\gamma\int_{\mathbb{R_{+}}}I(t-\rho)\Phi
(\rho)d\rho.$\newline For simplicity $\tau<\rho$, from which it follows that
$I(t-\rho)$ is positive for all $\rho\in\mathbb{R_{+}}$ by the above proof in
\ref{sec:positivityAndBoundedness}, if $\rho<\tau$ one proves positivity of
$S$ first then the positivity of $I$. The above equation therefore imply
\[
\frac{dS}{dt}\geq-\beta(t)I(t)S(t).
\]
By multiplying with the integrating factors for both sides, we get
\[
\frac{dS}{dt}e^{\int_{0}^{t}\beta(t)I(t)ds}+\beta(t)I(t)S(t)e^{\int_{0}%
^{t}\beta(t)I(t)ds}\geq0.
\]
or equivalently%
\[
\frac{d}{dt}\left(  S(t)e^{\int_{0}^{t}\beta(t)I(t)ds}\right)  \geq0.
\]
Integrate both sides:
\[
S(t)\geq S(0)e^{-\int_{0}^{t}\beta(t)I(t)ds}>0.
\]
Thus, it is clear that the solution $S(t)$ is positive since the initial value
$S(0)$ and the exponential functions are always positive. $\mathbb{\newline}%
$Thus the nonnegativity of $S(t)$ and $I(t)$ are proved for all $t\geq
-\tau>-\rho.$\newline\textbf{Positivity of }$R^{T}$:\newline From the given
System \ref{model:mainGenNoBirthOrNaturalDeath}, we get $R^{\text{T}}(t)=0$,
for all $t$, if the parameter $p=0$.\newline For a fixed $p\in(0,1]$. Recall
that $\operatorname*{supp}(\Phi)$ is a subset of $[\epsilon,M]$.\newline Now%
\begin{align*}
\frac{dR^{\text{T}}(t)}{dt}  &  =p\gamma I(t)-p\gamma\int_{0}^{\infty}%
I(t-\rho)\Phi(\rho)d\rho,\\
&  =p\gamma\left[  I(t)-\int_{\epsilon}^{M}I(t-\rho)\Phi(\rho)d\rho\right]  .
\end{align*}
By integrating both sides, we get
\[
R^{\text{T}}(t)=R^{\text{T}}(0)+p\gamma\int_{0}^{t}\int_{\epsilon}^{M}\left(
I(s)-I(s-\rho)\right)  \Phi(\rho)d\rho ds.
\]
The Fubini's theorem now gives%
\begin{align*}
R^{\text{T}}(t)  &  =R^{\text{T}}(0)+p\gamma\int_{\epsilon}^{M}\Phi(\rho
)\int_{0}^{t}\left(  I(s)-I(s-\rho)\right)  dsd\rho,\\
&  =R^{\text{T}}(0)+p\gamma\int_{\epsilon}^{M}\Phi(\rho)\left(  \int_{0}%
^{t}I(s)ds-\int_{-\rho}^{t-\rho}I(s)ds\right)  d\rho,\\
&  =R^{\text{T}}(0)+p\gamma\int_{\epsilon}^{M}\Phi(\rho)\left(  \int_{t-\rho
}^{t}I(s)ds-\int_{-\rho}^{0}I(s)ds\right)  d\rho.
\end{align*}
Including the initial data, Equation \eqref{eq:Rs}, yields%
\begin{align*}
R^{\text{T}}(t)  &  =c_{I}p\gamma\int_{\epsilon}^{M}\Phi(\rho)\rho
d\rho+p\gamma\int_{\epsilon}^{M}\Phi(\rho)\int_{t-\rho}^{t}I(s)dsd\rho
-c_{I}p\gamma\int_{\epsilon}^{M}\Phi(\rho)\rho d\rho,\\
&  =p\gamma\int_{\epsilon}^{M}\int_{t-\rho}^{t}\Phi(\rho)I(s)dsd\rho\geq0.
\end{align*}
Since $I(t)$ is positive, we have proved the positivity of $R^{\text{T}}(t)$
for all $t$.\newline\textbf{Positivity of }$R^{\text{P}}$:\newline The
dynamics for $R^{\text{P}}$ is given by
\[
\frac{dR^{\text{P}}(t)}{dt}=(1-p)\gamma I(t)\geq0,
\]
hence $R^{\text{P}}(t)$ is positive for all $t$.

\subsection{Proof of Corollary \ref{cor:continuous}\label{app:proofCor}}

\begin{proof}
Let $\varepsilon>0$ be given. Pick any bounded sequence $t_{n}\rightarrow t$.
We have%
\begin{align*}
F(t)  &  =\int_{a}^{t}H(t-\rho)\Phi(\rho)d\lambda(\rho),\\
&  =\underbrace{\int_{a}^{t_{n}}H(t_{n}-\rho)\Phi(\rho)d\lambda(\rho
)}_{=F(t_{n})}\\
&  +\underbrace{\int_{a}^{t_{n}}(H(t-\rho)-H(t_{n}-\rho))\Phi(\rho
)d\lambda(\rho)}_{=I(t_{n})}+\underbrace{\int_{t_{n}}^{t}H(t-\rho)\Phi
(\rho)d\lambda(\rho)}_{II(t_{n})},
\end{align*}
Now, since
\[
b\underset{\text{def.}}{=}\sup\{t_{n}\}<\infty
\]
we get by H\"{o}lder's inequality we get%
\begin{align*}
|I(t_{n})|  &  \leq\left\Vert (H(t-.)-H(t_{n}-.))\Phi(.)\right\Vert
_{L^{1}([a,b])}\\
&  \leq\left\Vert (H(t-.)-H(t_{n}-.))\right\Vert _{L^{\infty}([a,b])}%
\left\Vert \Phi\right\Vert _{L^{1}([a,b])}\\
&  \leq C\left\Vert (H(t-.)-H(t_{n}-.))\right\Vert _{L^{\infty}([a,b])},
\end{align*}
for some non-negative constant $C<\infty$. Since $H\in C(\mathbb{R})$ it is
uniformly continuous on $[a-b,b]$ and since $t_{n}\rightarrow t$ there is an
$N_{1}$ such that
\[
\left\Vert (H(t-.)-H(t_{n}-.))\right\Vert _{L^{\infty}([a,b])}<\frac
{\varepsilon}{2C},
\]
for all $n>N_{1}$, that is%
\[
|I(t_{n})|<\varepsilon/2.
\]
Furtheremore, since $H(t-.)\Phi(.)\in L^{1}(\mathbb{R})$ and $t_{n}\rightarrow
t$ there is an $N_{2}$ such that%
\[
|II(t_{n})|\leq\left\Vert H(t-.)\Phi(.)\right\Vert _{L^{1}([t,t_{n}%
])}<\varepsilon/2
\]
for all $n>N_{2}$. That is, for any $n>\max(N_{1},N_{2})$ we get
\[
\left\vert F(t)-F(t_{n})\right\vert \leq\left\vert I(t_{n})\right\vert
+\left\vert II(t_{n})\right\vert <\varepsilon/2+\varepsilon/2=\varepsilon.
\]
We have proved that $\lim_{n\rightarrow\infty}F(t_{n})=F(t)$, and since the
sequence $\{t_{n}\}$ was arbitrary, the function $F$ is continuous at $t$.
\end{proof}

\end{document}